\documentclass[12pt]{amsart}
\usepackage[utf8]{inputenc}
\usepackage[all]{xy}
\usepackage{amscd, amsmath, mathrsfs, amssymb, amsthm, amsxtra, bbding, epsfig, eucal, eufrak, graphicx, latexsym, url}
\usepackage{color,comment,bbm}
\def \C {{\mathbb C}}

\def \N {{\mathbb N}}

\def \R {{\mathbb R}}

\def \Z {{\mathbb Z}}

\def \d {\,{\rm d}}
\def\re{{\Re e\,}}

\def\le{\leqslant}

\def\ge{\geqslant}
\def\geq{\geqslant}

\def \fa{{\mathfrak{a}}}

\def \ff{{\mathfrak{f}}}
\DeclareMathOperator{\ee}{e}
\newcommand*{\dd}{%
  \mathop{\mathrm{d}\null}\mskip-\thinmuskip\mathord{\null}}
\newcommand*{\ic}{\mathrm{i}}

\newcommand*{\f}{\mathfrak{f}}
\newcommand*{\g}{\mathfrak{g}}
\newcommand*{\fh}{\mathfrak{h}}
\newcommand{\vep}{\varepsilon}

\newcommand*{\cush}{\mathfrak{S}_{\ell+1/2}}%

\theoremstyle{plain}
\newtheorem{theorem}{Theorem}

\newtheorem{lemma}{Lemma}[section]
\newtheorem{proposition}{Proposition}

\theoremstyle{remark}
\newtheorem{remark}{Remark}

\theoremstyle{definition}

\numberwithin{equation}{section}

\newcommand*{\pk}{\mathscr{H}}

\newcommand{\bpm}{\begin{pmatrix}}
\newcommand{\epm}{\end{pmatrix}}
\newcommand{\bsm}{\lt(\begin{smallmatrix}}
\newcommand{\esm}{\end{smallmatrix}\rt)}
\newcommand{\lt}{\left}
\newcommand{\rt}{\right}

\begin{document}

\vglue 0mm
%\vskip 5mm

\title[Fourier coefficients of modular forms]
{On Fourier coefficients of modular forms of half integral weight at squarefree integers}
\author{Y.-J. Jiang,  Y.-K. Lau, G.-S. L\"u, E. Royer \& J. Wu}

\address{
Yujiao Jiang
\\
Department of Mathematics
\\
Shandong University
\\
Jinan, Shandong 250100
\\
China}
\email{yujiaoj@hotmail.com}

\address{
Yuk-Kam Lau
\\
Department of Mathematics
\\
The University of Hong Kong
\\
Pokfulam Road
\\
Hong Kong}
\email{yklau@maths.hku.hk}

\address{
Guangshi L\"u
\\
Department of Mathematics
\\
Shandong University
\\
Jinan, Shandong 250100
\\
China}
\email{gslv@sdu.edu.cn}

\address{%
Emmanuel Royer\\
Clermont Universit\'e\\
Universit\'e Blaise Pascal\\
Laboratoire de math\'ematiques\\
BP 10448\\
F-63000 Clermont-Ferrand\\
France %
}
\curraddr{%
Emmanuel Royer\\
Universit\'e Blaise Pascal\\
Laboratoire de math\'ematiques\\
Les C\'ezeaux\\
BP 80026\\
F-63171 Aubi\`ere Cedex\\
France %
}
\email{{emmanuel.royer@math.univ-bpclermont.fr}}

\address{%
Jie Wu\\
CNRS\\
Institut \'Elie Cartan de Lorraine\\
UMR 7502\\
F-54506 Van\-d\oe uvre-l\`es-Nancy\\
France}
\curraddr{%
Université de Lorraine\\
Institut \'Elie Cartan de Lorraine\\
UMR 7502\\
F-54506 Van\-d\oe uvre-l\`es-Nancy\\
France
}
\email{jie.wu@univ-lorraine.fr}

\date{\today}

\begin{abstract}
We show that the Dirichlet series associated to the Fourier coefficients of a half-integral weight Hecke eigenform at squarefree integers extends analytically to a holomorphic function in the half-plane $\re s>\tfrac{1}{2}$. This exhibits a high fluctuation of the coefficients at squarefree integers.
\end{abstract}

\maketitle

\addtocounter{footnote}{1}

\section{Introduction}

Some modular forms are endowed with nice arithmetic properties, for which techniques in  analytic number theory can be applied to unveil their  extraordinary features. For instance, Matom\"aki and Radziwill \cite{MR} made an important progress for multiplicative functions with an application (amongst many) to give a very sharp result on the holomorphic Hecke cusp eigenforms of integral weight. The Hecke eigenforms of half-integral weight is substantially different from the case of integral weight. A simple illustration is the multiplicativity of their Fourier coefficients. If $f$ is a Hecke eigenform of integral weight (for $SL_2(\Z)$),  its Fourier coefficient $a_f(m)$ will be factorized into $a_f(m)= \prod_{p^r\|m} a_f(p^{r})$. However, for a Hecke eigenform $\f$ of half-integral weight (for $\Gamma_0(4)$), we only have $a_\f(tm^2)= a_\f(t) \prod_{p^r\|m} a_\f(p^{2r})$ for any squarefree  $t$, due to Shimura. (Both $a_f(1)=a_\f(1)=1$ are assumed.)  This, on one hand,  alludes to the mystery of $\{a_\f(t)\}_{t\ge 1}^{\flat}$\footnote{\mbox{}  The superscript $\mbox{}^\flat$ is to  indicate that the index is supported at squarefree integers.} and,  on the other hand, provides an interesting object $\{a_\f(n)\}_{n\ge 1}$ whose multiplicativity (is limited to the square factors) has no analogue to the classical number-theoretic functions. 

The classical divisor function $\tau(n):=\sum_{d|n} 1$ appears to be Fourier coefficients of some Eisenstein series.  
In the literature there are investigations on $\{\tau(t)\}_{t\ge 1}^\flat$ and 
on the associated Dirichlet series $L(s) := \sum_{t\ge 1}^\flat \tau(t) t^{-s}$, which is however rather obscure. 
Using the multiplicative properties of $\tau(n)$, 
$L(s)$ is connected to the reciprocal of the Riemann zeta-function $\zeta(2s)^3$, 
and it extends analytically to (a slightly bigger region containing) the half-plane $\re s> \tfrac{1}{2}$ 
with exactly one double pole at $s=1$. 
A further extension is equivalent to a progress towards the Riemann Hypothesis. 

On the other hand,  
to study the sign-changes in $\{a_\f(t)\}_{t \ge 1}^\flat$, Hulse et al. \cite{HKKL2012} recently considered 
$L_\f^\flat(s) := \sum_{t\ge 1}^{\flat} \lambda_\f(t) t^{-s}$ (where $\lambda_\f(t)=a_\f(t)t^{-(\ell/2-1/4)}$). 
Interestingly they showed that $L_\f^\flat(s)$ extends analytically to a holomorphic function in $\re s>\tfrac{3}{4}$.
%yielding easily that $\sum_{t\le x}^\flat \lambda_\f (t)\ll x^{3/4+\vep}$ holds {\it on average}.  
One naturally asks how far $L_\f^\flat(s)$ can further extend to. Compared with the case of $\tau(n)$ but without adequate multiplicativity,  a  continuation to the region $\re s>\tfrac{1}{2}$ is curious, non-trivial and plausibly (very close to) the best attainable with current technology. %This is also related to the problem of square-root cancellation of $\{a_\f(t)\}_{t \ge 1}^\flat$.

The argument of proof in \cite{HKKL2012} is based on the convexity principle and includes two key ingredients: 
\begin{itemize}
\item[(a)] 
the inequality $\lambda_\f(tr^2)\ll_{\vep} |\lambda_\f(t)| r^\vep$, 
\item[(b)]
the functional equations of the twisted $L$-functions for $\f$ by additive characters $\ee(un/d)$. 
\end{itemize}
The inequality (a) is a substitute for the unsettled Ramanujan Conjecture for half-integral weight Hecke eigenforms, 
and this is derived from the Shimura correspondence and the Deligne bound for modular forms of integral weight. According to various $d$'s, the functional equations of (b) involves the Fourier expansions of $\f$ at different cusps, which is detailedly computed in \cite{HKKL2012}. However, due to the multiplier system, the Fourier expansion at the cusp $\tfrac{1}{2}$ is not of period $1$, of which Hulse et al seemed not aware. We shall propose an amendment in Section~\ref{S4}. 

Our main goal  is to prove that $L_\f^\flat(s)$ extends analytically to $\re s>\tfrac{1}{2}$.  
We shall not use the convexity principle but apply the approximate functional equation with the point $s$ 
close to the line $\re s=\tfrac{1}{2}$ (from right).  
The cancellation amongst the exponential factors and real quadratic characters arising from the twisted $L$-functions are explored. Without a known Ramanujan Conjecture, 
the inequality (a) is crucial and indeed we need more -- an inequality of the same type for the Fourier coefficients at all cusps, which is done in Section~\ref{S3}. There we study the Fourier coefficients of a (complete) Hecke eigenform 
at the two cusps $0$ and $\tfrac{1}{2}$, 
and derive some inequalities and bounds useful for analytic approaches, which are  of their own interest.

\vskip 8mm

\section{Main results}

Let \(\ell\geq 2\)  be a positive integer, and denote by \(\cush\) the set of all holomorphic cusp forms of weight \(\ell+1/2\) 
for the congruence subgroup \(\Gamma_0(4)\). The Fourier expansion of \(\f\in\cush\) at \(\infty\) is written as
\begin{equation}\label{eq_newone}%
\f(z)=\sum_{n\ge 1} \lambda_{\f}(n) n^{\ell/2-1/4}\ee(nz) 
\qquad (z\in\pk),
\end{equation}
where \(\ee(z) = \ee^{2\pi\ic z}\) and \(\pk\) is the Poincar\'e upper half plane. 
Define 
\begin{equation}\label{def_Mfs}
L_\f^\flat(s) := \sideset{}{^\flat}\sum_{t\ge 1} \lambda_\f(t) t^{-s}
\end{equation}
for $s=\sigma+\text{i}\tau$ with $\sigma>1$,
where $\sum_{t\ge 1}^\flat$ ranges over squarefree integers $t\ge 1$.

\begin{theorem}\label{main}
Let $\ell\ge 2$ be a positive integer and 
${\mathfrak f}\in \cush$ be a complete Hecke eigenform.
The series $L_\f^\flat(s)$ in \eqref{def_Mfs} extends analytically to a holomorphic function on $\Re e\, s>\tfrac{1}{2}$.
Moreover, for any $\varepsilon>0$ we have
\begin{equation}\label{UBMs}
L_\f^\flat(s)
\ll_{{\mathfrak f}, \varepsilon} (|\tau|+1)^{1-\sigma+2\varepsilon}
\qquad
(\tfrac1{2}+\varepsilon\le\sigma\le 1+\vep, \tau\in \R),
\end{equation}
where the implied constant depends on ${\mathfrak f}$ and $\varepsilon$ only.
\end{theorem}

\begin{remark} It follows immediately the Riesz mean $\sum_{t\le x}^\flat (1-t/x)\lambda_\f (t) \ll x^{1/2+\vep}$, exhibiting a support towards square-root cancellation of $\{\lambda_\f(t)\}_{n\ge 1}^\flat$. 

An application of Theorem~\ref{main} is  a  better lower bound (than \cite[Theorem 4]{LauRoyerWu2014}) for the sign-changes of $\{\lambda_\f(t)\}^\flat_{t\ge 1}$ with $t\in [1,x]$ 
and the other is to consider mean value of $\lambda_\f(t)$, which will be done in other papers. 
\end{remark}

\vskip 8mm

\section{Half-integral weight cusp forms for $\Gamma_0(4)$}\label{S3}

We follow Shimura \cite{Shimura1973}  to explicate the definition of \(\f \in \cush\).  The main aim is to discuss some properties of the Fourier coefficients at {\it all} cusps  when $\f$ is a complete Hecke eigenform.

Let $GL_2^+(\R)$ be the set of all real $2\times 2$ matrices with positive determinant. Define $\widetilde{G}$ to be the set of all $(\alpha, \varphi(z))$ where $\alpha = \Big(\begin{matrix}a & b\\ c & d\end{matrix}\Big)\in GL_2^+(\R)$  
and $\varphi(z)$ is a holomorphic function on $\pk$ such that
\[
\varphi(z)^2 := \varsigma \det(\alpha)^{-1/2}(cz+d), 
\quad 
\mbox{for some $\varsigma\in \C$ with $|\varsigma|=1$}.
\]
Then $\widetilde{G}$ is a group under the composition law
\(
(\alpha, \varphi(z))(\beta, \psi(z))= (\alpha\beta, \varphi(\beta z) \psi(z)).
\)
The projection map $(\alpha, \varphi(z))\mapsto \alpha$ is a surjective homomorphism 
from $\widetilde{G}$ to $GL_2^+(\R)$. We write $(\alpha,\varphi(z))_*=\alpha$. 
Let $f$ be any complex-valued function on $\pk$.  
The slash operator $\xi\mapsto f\vert_{[\xi]}$, defined as  
\[
f\vert_{[\xi]} := \varphi(z)^{-(2\ell+1)} f(\alpha z) 
\quad
\text{if}\;\,
\xi = (\alpha, \varphi(z)),
\] 
gives an anti-homomorphism on $\widetilde{G}$, i.e. $f\vert_{[\xi\eta]}= (f\vert_{[\xi]})\vert_{[\eta]}$.

Define for $\gamma = \Big(\begin{matrix}a & b\\ c & d\end{matrix}\Big)\in \Gamma_0(4)$ and $z\in \pk$, 
$$
j(\gamma,z) :=\frac{\theta(\gamma z)}{\theta(z)} 
= \vep_d^{-1} \bigg(\frac{c}{d}\bigg)(cz+d)^{1/2},
$$
where $\vep_d = 1$ or $\ic$ according as $d \equiv 1$ or $3$ $({\rm mod}\,4)$, the extended Jacobi symbol $\big(\frac{c}{d}\big)$  and the square root $(cz+d)^{1/2}$ are  defined  as in \cite{Shimura1973}. The map \(\gamma\mapsto  \gamma^*\)  with 
\(
\gamma^*:= (\gamma, j(\gamma,z))
\)
is an one-to-one homomorphism from $\Gamma_0(4)$ to $\widetilde{G}$. For $\gamma\in \Gamma_0(4)$, we will abbreviate $f\vert_{[\gamma^*]}$ as $f\vert_{[\gamma]}$.

A cusp form $\f$ of weight $\ell+1/2$ for $\Gamma_0(4)$ is a  holomorphic function on $\pk$ such that 
\begin{enumerate}
\item[$1^\circ$] 
$\f|_{[\gamma]} = \f$ for all $\gamma\in \Gamma_0(4)$,
\item[$2^\circ$] 
$\f$ admits a Fourier series expansion at every cusp $\fa\in \{0,-\tfrac{1}{2}, \infty\}$, 
\[
\f\vert_{[\rho]}=\sum_{\substack{n\in \Z\\ n+r>0}} c_n \ee((n+r)z).
\]
Here  $\rho\in \widetilde{G}$ satisfies that its projection is a scaling matrix for the cusp $\fa$, i.e. $\rho_*(\infty)=\fa$, and for some $|t|=1$, 
$$
\rho^{-1} \eta^* \rho = \left(\eta_\infty, t\right)
\quad\text{with}\;\,
\eta_\infty:=\begin{pmatrix} 1 & 1 \\ & 1\end{pmatrix},
$$
where $\eta$ is a generator of the stabilizer $\Gamma_\fa$ in $\Gamma_0(4)$ for the cusp $\fa$.
The value of 
$r\in [0,1)$ is determined by $\ee(r)=t^{2\ell+1}$. (See \cite[p.444]{Shimura1973}.)
\end{enumerate}

\subsection{Fourier expansions at the three cusps}
Explicitly we take $\rho=\rho_\fa$ where
\begin{eqnarray}
\rho_\fa= 
\left\{
\begin{array}{ll}
\left( \begin{pmatrix} 1 &  \\  & 1\end{pmatrix}, 1\right)  & \mbox{ for $\fa=\infty$},\vspace{1mm}
\\\noalign{\vskip 0,5mm}
\left( \begin{pmatrix} 1 &  \\ -2 & 1 \end{pmatrix}, (-2z+1)^{1/2}\right) & \mbox{ for $\fa=-\frac12$}, \vspace{1mm}
\\\noalign{\vskip 0,5mm}
\left( \begin{pmatrix} & -1 \\ 4 & \end{pmatrix}, 2^{1/2}(-\ic z)^{1/2}\right) & \mbox{ for $\fa=0$}.
\end{array}
\right.
\end{eqnarray}
Set \(\eta_\fa = {\rho_\fa}_* \eta_\infty {\rho_\fa}_*^{-1}\).  
Then $\eta_\fa \in \Gamma_0(4)$ for all the three cusps. 
A direct checking shows that  \(\rho_\fa^{-1}  \eta_\fa^*\rho_\fa =  \left(\eta_\infty, t_\fa\right)\) 
where $t_\fa = 1,\ic, 1$ for $\fa=\infty, -\tfrac{1}{2}, 0$, respectively. 
(When $\fa=-\tfrac{1}{2}$, the factor $\vep_{-1}^{-1}\left(\frac{-4}{-1}\right)$ inside $j(\eta_\fa^*,z)$ equals $\ic$.)
Hence, for \(\f \in \cush\),  $\f(z+1)=\f(z)$ (note $\f\vert_{[\rho_\infty]}=\f$) and $\f\vert_{[\rho_0]}(z+1)=\f\vert_{[\rho_0]}(z)$, while for $\fa=-\tfrac{1}{2}$,
\begin{eqnarray}\label{gperiod}
\f\vert_{[\rho_{\fa}]}(z+1) 
= t_{-1/2}^{2\ell +1}  \f\vert_{[\rho_{\fa}\eta_\infty]}(z) 
= \ic^{2\ell +1} \big(\f\vert_{[\eta_{\fa}^*]}\big)\vert_{[\rho_{\fa}]}(z) 
=\ic^{2\ell +1} \f\vert_{[\rho_{\fa}]}(z).
\end{eqnarray} 

Let $\alpha= \left(\begin{pmatrix} 4 & \\ & 1\end{pmatrix}, 2^{-1/2}\right)$. For our purpose, we set 
\begin{equation}\label{ghdef}
\g(z) := \big(\f \vert_{[\rho_{-1/2}]}\big)\vert_{[\alpha]} (z) =2^{\ell+1/2} \f \vert_{[\rho_{-1/2}]} (4z)
\quad \mbox{ and } \quad 
\fh(z) := \f\vert_{[\rho_0]}(z).
\end{equation}
Their Fourier series expansions (at $\infty$) are of the form
\begin{equation}\label{g-fourier}
\begin{aligned}
\g(z) 
& = 2^{\ell+1/2} \sum_{n\ge 0} c_n \ee\big((4 n+(2+(-1)^{\ell-1}))z\big)
\\
& = 2^{\ell+1/2} \sum_{n\ge 1} \lambda_{\g}(n) n^{\ell/2-1/4}\ee(nz), 
\quad
\text{say,}
\end{aligned}
\end{equation}
where the sequence $\{\lambda_{\g}(n)\}$ is supported on positive integers $n\equiv (-1)^{\ell}\,(\bmod\,{4})$,  and 
\begin{equation}\label{h-fourier}
\fh(z) 
= \sum_{n\ge 1} \lambda_{\fh}(n) n^{\ell/2-1/4}\ee(nz).
\end{equation}
\begin{remark}\label{rmki} (i) The cusp form $\fh(z)$ is $\f_0(z)$ in \cite{HKKL2012} but $\g(z) = \f_{\frac12} (4z)$, not $\f_{\frac12}(z)$, there.  The Fourier expansion of $\f_{\frac12}(z)$  at $\infty$ is  of the form $\sum_{n\ge 1} c_n \ee((n+\frac14)z)$. 
\par
(ii)  The form $\fh$ is a cusp form for $\Gamma_0(4)$ but $\g$ is a cusp form for $\Gamma_0(16)$. 
\par
(iii) Using the Rankin-Selberg theory, one can prove that 
\begin{equation}\label{msq}
\sum_{n\le x} |\lambda_f (n)|^2\sim x \qquad \mbox{($f=\f,\g$ or $\fh$).}
\end{equation}
See \cite[Section 3]{LauRoyerWu2014}, for example. (There the assumption that $\f$ is a complete Hecke eigenform is not necessary, which is clearly seen from the proof.) 
\end{remark}

\subsection{Eigenform properties of a complete Hecke eigenform at various cusps}

Let $N$ be a positive integer divisible by $4$, and $p\nmid N$ be any prime. 
The action of the Hecke operator $\mathsf{T}(p^2)$ on a modular form $f$ of half-integral weight $\ell+1/2$ 
for $\Gamma_0(N)$ is defined as (cf. \cite[p.451]{Shimura1973})
\begin{eqnarray*}
\mathsf{T}(p^2) f 
:=   p^{\ell -3/2} \Big\{\sum_{0\le b<p^2} f\vert_{[\alpha_b^\star]} + \sum_{1\le h<p} f\vert_{[\beta_h^\star]} + f\vert_{[\sigma^\star]}\Big\},
\end{eqnarray*}
where 
\begin{align*}
\alpha_b^\star 
& := \big(\alpha_b, p^{1/2}\big) 
= \left(\begin{pmatrix} 1 & b \\ & p^2\end{pmatrix}, p^{1/2}\right),
\\
\beta_h^\star 
& := \big(\beta_h, \vep_p^{-1} \big(\tfrac{-h}{p}\big)\big)
= \left(\begin{pmatrix} p & h \\ & p\end{pmatrix},\vep_p^{-1} \left(\frac{-h}p\right)\right),
\\
\sigma^\star 
& := \big(\sigma, p^{-1/2}\big) 
= \left(\begin{pmatrix} p^2 &  \\ & 1\end{pmatrix}, p^{-1/2}\right).
\end{align*}
Suppose $\f$ is a complete Hecke eigenform, i.e. $\mathsf{T}(p^2)\f =\omega_p \f$ for all prime $p$.  One may wonder whether  $\g$ and $\fh$ defined as in \eqref{ghdef} are eigenforms. We can prove the following. 
\begin{lemma}\label{lem-eigen} Let $p$ be any odd prime. If $\f$ is an eigenform of $\mathsf{T}(p^2)$, then so are the forms $\g$ and $\fh$ defined in \eqref{ghdef} and both have the same eigenvalues as $\f$. 
\end{lemma}
\begin{proof} Let $N=4$ or $16$, and $\Delta_0=\Gamma_0(N)^*$ be the image of $\Gamma_0(N)$ under the lifting map.  It suffices to show that 
for (i) $\rho_{-1/2}\alpha$, $N=16$ and (ii) $\rho=\rho_0$, $N=4$, the elements $\rho \alpha_b^\star, \rho\beta_h^\star ,\rho\sigma^\star$ ($0\le b<p^2$, $1\le h<p$) form  a set  of representatives for 
\begin{eqnarray*}
\Delta_0\backslash \Big(\Delta_0 \sigma^\star \rho \sqcup \bigsqcup_{1\le h<p^2} \Delta_0 \beta_b^\star \rho \sqcup \bigsqcup_{0\le b<p^2} \Delta_0 \alpha_b^\star \rho \Big).
\end{eqnarray*}

(i) 
For the case $\rho=\rho_{-1/2}\alpha$, we check by routine calculation that 
\begin{align*}
\Delta_0\rho \sigma^\star
& = \Delta_0 \alpha_{(p^2+1)/2}^\star\rho, 
\\ 
\Delta_0 \rho \alpha_{(p^2-1)/8}^\star
& = \Delta_0\sigma^\star \rho, 
\\
\{\Delta_0 \rho\beta_h^\star \,:\, 1\le h<p\}
& = \{\Delta_0 \alpha_d^\star \rho \,:\, p\,\|\,(1-2d)\}, 
\\
\{\Delta_0 \rho \alpha_b^\star \,:\, p\nmid (1+8b)\}
& = \{\Delta_0 \alpha_d^\star\rho \,:\, p\nmid 1(-2d)\}, 
\\
\{\Delta_0\rho \alpha_d^\star \,:\, p\,\|\,(1+8b)\}
& = \{\Delta_0\beta_h^\star\rho \,:\, 1\le h<p\}.
\end{align*}
For example,  from
$$
\rho=\rho_{-1/2}\alpha = \left(\begin{pmatrix} 4 & \\ -8 & 1 \end{pmatrix}, 2^{-1/2}(-8z+1)^{1/2}\right)
$$ 
we obtain  $\rho_* \beta_h \rho_*^{-1}=  \gamma\alpha_{d}$, where
$$
\gamma
= \begin{pmatrix} 
p+8h & (4h-d(p+8h))p^{-2}
\\\noalign{\vskip 2mm}
-16h & (p-8h+16hd)p^{-2}
\end{pmatrix} 
\in \Gamma_0(16)
$$
if we take $1\le d<p^2$ such that $d(p+8h)\equiv 4h\,(\bmod\,{p^2})$. 
Note that this choice implies $p-8h+16hd\equiv p(1-2d)\,(\bmod\,{p^2})$ and $dp\equiv 4h(1-2d)\,(\bmod\,{p^2})$. 
The latter implies $p\mid (1-2d)$, so the former is $\equiv 0\,(\bmod\,{p^2})$. 
Next the $\varphi$-part of $\rho\beta_h^\star \rho^{-1}$ is 
$$
\vep_p^{-1} \left(\frac{-h}p\right) (-16hz+p-8h)^{1/2} p^{-1/2}.
$$
To evaluate the $\varphi$-part of $\gamma^*\alpha_d^\star$,  
we remark that $j(\gamma, z)= j(\gamma^{-1}, \gamma z)^{-1}$ 
and thus consider ${\gamma^*}^{-1}$ whose $j$-part is simply 
$$
\vep_{p+8h}^{-1} \left(\frac{16h}{p+8h}\right) (16hz+p+8h)^{1/2}.
$$
Hence, the $\varphi$-part of $\gamma^*\alpha_d^\star$ is
$$
\vep_p \left(\frac{h}{p} \right)(-16\gamma \alpha_d z+ p+8h)^{-1/2} p^{1/2}.
$$
From $\gamma\alpha_d = \rho_*\beta_h\rho_*^{-1}$ and $\vep_p \big(\tfrac{-1}p\big) = \vep_p^{-1}$, 
we easily verify this case. 
The other cases are checked in the same way.

\par
(ii) For the case $\rho=\rho_0$, we find similarly that
\begin{align*}
\Delta_0\rho \sigma^\star
& = \Delta_0 \alpha_{0}^\star\rho, 
\\ 
\Delta_0 \rho \alpha_{0}^\star
& = \Delta_0\sigma^\star \rho, 
\\
\{\Delta_0 \rho \beta_h^\star \,:\, 1\le h<p\}
& = \{\Delta_0 \alpha_{pd}^\star \rho \,:\, 1\le d<p\}, 
\\
\{\Delta_0\rho \alpha_d^\star \,:\, p\,\|\,b\}
& = \{\Delta_0\beta_h^\star\rho \,:\, 1\le h<p\},
\\
\{\Delta_0 \rho \alpha_b^\star \,:\, p\nmid b\}
& = \{\Delta_0 \alpha_d^\star\rho \,:\, p\nmid d\}.
\end{align*}
\end{proof}

\subsection{Shimura's correspondence and bounding coefficients}
Let $\f\in \cush$, not necessarily a complete Hecke eigenform. By Shimura's theory \cite[Section 3]{Shimura1973}, for any squarefree $t\ge 1$, there is a cusp form ${\rm Sh}_t \f$ of weight $2\ell$ for $\Gamma_0(2)$ such that 
\begin{equation}\label{Sl}
t^{\ell/2-1/4} L(s+\tfrac12, \chi_t) \sum_{n\ge 1} \lambda_\f(tn^2) n^{-s} = L(s, {\rm Sh}_t \f),
\end{equation}
where $L(\cdot,\chi_t)$ is the Dirichlet $L$-function associated to the character 
$$
\chi_t(n) = \chi_0(n) \bigg(\frac{-1}n\bigg)^\ell \bigg(\frac{t}n\bigg)
$$
($\chi_0$ is the principal character mod $4$) and 
$L(s, F) := \sum_{n\ge 1} \lambda_F(n)n^{-s}$ is the $L$-function for the cusp form of integral weight $2\ell$ with nebentypus $\chi_0^2$,
$$
F(z) = \sum_{n\ge 1} \lambda_F(n)n^{(2\ell-1)/2} \ee(nz).
$$

The Shimura lift $\f\mapsto {\rm Sh}_t\f$ commutes with Hecke operators:  $ {\rm Sh}_t (\mathsf{T}(p^2)\f)= T(p) ({\rm Sh}_t \f) $ for all primes $p$.\footnote{\mbox{}  This commutativity is pointed out in \cite[Corollary 3.16]{Ono2004} under the extra condition $p\nmid 4tN$, which is relaxed to all primes $p$ in \cite{Purkait2013}.} 
It follows that  the coefficients $\lambda_\f(mp^{2r})$ satisfy a recurrence relation in $r$  
when $\f$ is a $\mathsf{T}(p^2)$-Hecke eigenform. 
Moreover, if $\f$ is a Hecke eigenform  of $\mathsf{T}(p^2)$ for all $p\notin \mathcal{S}$ 
(where $\mathcal{S}$ is any set of primes), the right-hand side of \eqref{Sl} will admit a factorization 
(see Corollary 1.8 and Main Theorem in \cite{Shimura1973})
\begin{equation}\label{Shimuraep}
t^{\ell/2-1/4} L(s+\tfrac12, \chi_t) \sum_{n\ge 1} \frac{\lambda_\f(tn^2)}{n^s} 
= \sum_{\substack{n\ge 1\\ p\nmid n \Rightarrow p\in \mathcal{S}}} \frac{\lambda_{{\rm Sh}_t\f}(n)}{n^s} 
\prod_{p\notin \mathcal{S}} \bigg(1-\frac{\omega_p}{p^s} + \frac{\chi_0(p) }{p^{2s}}\bigg)^{-1},
\end{equation}
where $\mathsf{T}(p^2) \f = \omega_p p^{(2\ell -1)/2} \f$. 
Remark that the product $\prod_{p\notin \mathcal{S}}$ remains the same for lifts of  different squarefree $t$'s. 

The commutativity between ${\rm Sh}_t$ and $\mathsf{T}(p^2)$ implies that  $\omega_p$ is also an eigenvalue of the Hecke operator $T(p)$ for ${\rm Sh}_t\f$. Decompose 
\begin{equation}\label{lc}
{\rm Sh}_t\f (z) = \sum_i c_i f_i(\ell_iz)
\end{equation}
where each $f_i$ is a newform (of perhaps lower level) and $f_i(\ell_i z)$'s are linearly independent. 
Let $\mathcal{S}'$ be the set of all prime $p$ dividing the level of ${\rm Sh}_t\f$,  
so $\mathcal{S}'=\{2\}$ in our case. If $p\notin \mathcal{S}'$, then $T(p)(f_i(\ell_i z)) = (T(p)f_i)(\ell_iz)$, $\forall$ $i$. 
(See \cite[(2.14)]{ILS2000} and \cite[Section 14.7]{IK2004}.) 
Applying $T(p)$ on both sides of \eqref{lc}, we thus see that $\omega_p$ is the $T(p)$-eigenvalue of some newform (for $p \nmid$ level of ${\rm Sh}_t\f$) and hence 
\begin{equation}\label{omegap}
|\omega_p|\le 2 \quad\mbox{ $\forall$ $p\notin \mathcal{S}\cup \mathcal{S}'$}, 
\end{equation}  
by Deligne's bound. Consequently we have the following estimate for \(\f\in \cush\).

\begin{lemma}\label{lsqfree} 
Let $\mathcal{Q}$ be a (not necessarily finite) set of primes with $2\in \mathcal{Q}$. 
Suppose $\f$ is an Hecke eigenform of $\mathsf{T}(p^2)$ for all $p\notin \mathcal{Q}$.  
Let $m\ge 1$ be any integer decomposed into $m=qr^2$ such that  $p\mid r$ implies $p\notin \mathcal{Q}$, 
and $p^2\mid q$ implies $p\in \mathcal{Q}$. Then we have
$$
|\lambda_\f(m)|\le |\lambda_\f(q)|\tau(r)^2.
$$
\end{lemma}

\begin{remark}\label{rmk2} 
Every integer $m\ge 1$ decomposes uniquely into the desired form: 
Decompose $m$ uniquely into $m= tn^2$ where $t$ is squarefree, 
write $n= ur$ such that $p\mid u$ implies $p\in \mathcal{Q}$ 
and $p\mid r$ implies $p\notin \mathcal{Q}$, and then set $q=tu^2$.
\end{remark}

\begin{proof} 
Let $m=qr^2=tu^2r^2$ be decomposed as in Remark~\ref{rmk2}. By \eqref{Shimuraep}, we see that
\begin{eqnarray*}\label{tmp1}
t^{\ell/2-1/4} \lambda_\f(tu^2r^2) 
= \bigg(\sum_{ab=u} \lambda_{{\rm Sh}_t\f}(a)\mu(b)\frac{\chi_t(b)}{\sqrt{b}}\bigg)
\bigg(\sum_{cd= r} \omega_c \mu(d) \frac{\chi_t(d)}{\sqrt{d}}\bigg),
\end{eqnarray*}
where $\omega_c$ is the coefficient of $c^{-s}$ 
in $\prod_{p\notin \mathcal{Q}} \left(1-{\omega_p}p^{-s} + \chi_0(p) p^{-2s}\right)^{-1}$
and $\mu(d)$ is the M\"obius function. 
The case $r=1$ tells that the first bracket is $t^{\ell/2-1/4} \lambda_\f(tu^2)$, i.e. $t^{\ell-1/4} \lambda_\f(q)$. 
Next, since $|\omega_c|\le \tau(c)$ (by \eqref{omegap} and its definition), 
the absolute value of the second bracket is $\le \tau(r)^2$. 
\end{proof}

\subsection{Bounds for coefficients of a complete Hecke eigenform at all cusps}\hspace{-2mm}\footnote{\mbox{} The content of this subsection is not used in  the remaining part of the paper but we would include here for its own interest and for applications in other occasions.}
In case $\f$ is a complete Hecke eigenform, we may express \eqref{Shimuraep} as 
\begin{equation}\label{Shimuralift}
t^{\ell/2-1/4} L(s+\tfrac12, \chi_t) \sum_{n\ge 1} \lambda_\f(tn^2) n^{-s} 
= t^{\ell/2-1/4} \lambda_\f(t) L(s, F),
\end{equation}
where the Shimiura lift $F$ is a cusp form independent of $t$. As the Ramanujan's conjecture holds for holomorphic newforms of integral weight, the $r$th Fourier coefficients of $F$ are $\ll_\f \tau(r) r^{\ell-1/2}$, where the implied constant is independent of $t$. Consequently the question of the size of $\lambda_\f(m)$ is reduced to the size at the squarefree part of $m$:  
\begin{eqnarray}\label{lb}
\lambda_\f(m)\ll_\f |\lambda_\f(t)| \tau(r)^2
\end{eqnarray}
  if $m=tr^2$ and squarefree $t$. Due to Iwaniec \cite{Iwaniec1987} or Conrey \& Iwaniec \cite{ConreyIwaniec2000}, etc, there are good estimates for 
\begin{equation}\label{varrho}
\lambda_\f (t)\ll_{\f,\varrho} t^\varrho \qquad \mbox{ $\forall$ squarefree $t$},
\end{equation}
for some $0<\varrho<\tfrac{1}{4}$. The value of $\varrho$ is $\frac16+\vep$ by \cite{ConreyIwaniec2000}. 

%Thus we have \lambda_\f(m)\ll t^\varrho \tau(r)^2 when $m=tr^2$ for some squarefree $t$. 

We know from Lemma~\ref{lem-eigen} that for a complete Hecke eigenform $\f$, the forms $\g$ and $\fh$ are eigenforms of $\mathsf{T}(p^2)$ with the same corresponding eigenvalue for all odd prime $p$. But for $p=2$, we do not get the same conclusion.  This may result in an unpleasant situation of without \eqref{lb}. Note that Lemma~\ref{lsqfree} gives at most a bound of the form $|\lambda_f(t2^{2j})|\tau(r)^2$ (where $f=\g,\fh$).   Now we attempt to clarify as much as possible. 

In view of \cite[Proposition 1.5]{Shimura1973}, 
the Hecke operator $\mathsf{T}(2^2)$ is the same as the operator $U_4$ whose action is
\begin{eqnarray}\label{u4}
(f\vert {U_4})(z)
= \frac14\sum_{\nu \,({\rm mod}\,4)} f\bigg(\frac{z+\nu}4\bigg) 
= \sum_{n\ge 1} a(4n) \ee(nz)
\end{eqnarray}
if $f (z)=\sum_{n\ge 1} a(n) \ee(nz)$. Then it follows easily that $\g\vert \mathsf{T}(2^2) = 0$, 
because by \eqref{u4} and \eqref{gperiod}, 
\begin{equation}\label{gu4}
(\g\vert U_4)(z) 
= 2^{\ell-3/2}\sum_{\nu\,({\rm mod}\,4)} \f\vert_{[\rho_{-1/2}]}(z+\nu)
= 2^{\ell-3/2}\f\vert_{[\rho_{-1/2}]}(z) \sum_{0\le \nu\le 3} \ic^{\nu(2\ell+1)}
\end{equation}
where the sum is obviously zero. Thus $\g$ is also a complete Hecke eigenform although it takes the different eigenvalue $0$ for $\mathsf{T}(2^2)$, implying the validity \eqref{lb} for $\g$ as well. 

However for the case of $\fh$, we cannot get the conclusion of $\mathsf{T}(2^2)$-eigenform and we shall get the analogous bound via some bypass.  To its end, let us recall Niwa's result in \cite{Niwa1977}, cf. Kohnen \cite[p. 250]{Kohnen1980}, saying that $U_4W_4$ is Hermitian operator on \(\cush\) and 
\begin{eqnarray}\label{u4w4}
U_4W_4U_4W_4-\mu U_4W_4  - 2\mu^2=0,
\end{eqnarray} 
where $\mu= \left(\frac2{2\ell+1}\right) 2^{\ell-1}= (-1)^{\ell(\ell+1)/2} 2^{\ell-1}$ and 
$$
(f\vert W_4)(z) =(-2\ic z)^{-(\ell+1/2)} f(-\tfrac1{4z}) = f\vert_{[\rho_0]}(z).
$$
Suppose $\f\vert \mathsf{T}(2^2) = c \f$ for some scalar $c$.\footnote{\mbox{}  Here we write $\f\vert \mathsf{T}(p^2)$ for $\mathsf{T}(p^2)\f$.} By \eqref{u4w4} and $U_4= \mathsf{T}(2^2)$,  we get
$$
c (\f \vert W_4U_4W_4)-c \mu (\f\vert W_4)  - 2\mu^2 \f=0.
$$
(Note that the operator acts on $\f$ from right.) In particular we observe that $c\neq 0$, because otherwise,  
$-2\mu^2 \f =0$ implying $\f=0$. As $\fh= \f\vert W_4$ and $W_4$ is an involution (i.e. $W_4^2$ is the identity), 
we deduce that
\begin{eqnarray}\label{hu4}
(\fh\vert U_4) = \mu \f + 2\mu^2 c^{-1} \fh.
\end{eqnarray}
We separate into two cases:
\begin{itemize}
\item Case 1: $c^2\neq 2 \mu^2$. 
\par
We set $\alpha := c \mu /(2\mu^2-c^2)$ and consider the form 
$\mathfrak{H}:= \fh +\alpha \f\in \cush$. Then $c \alpha+\mu = 2\alpha \mu^2/c$ and thus by \eqref{hu4},
$$
\mathfrak{H}\vert\mathsf{T}(2^2) = \mathfrak{H}\vert U_4 = 2 \mu^2 c^{-1} \mathfrak{H}.
$$
i.e. The cusp form $\mathfrak{H}$ is an eigenform of $\mathsf{T}(2^2)$, and by Lemma~\ref{lem-eigen}, $\mathfrak{H}$ is also an eigenform of $\mathsf{T}(p^2)$ for all odd primes $p$. (Note that $\f$ and $\fh$ have the same $\mathsf{T}(p^2)$-eigenvalue.) Consequently, both coefficients $\lambda_{\mathfrak{H}}(m)$ and $\lambda_\f(m)$ satisfy \eqref{lb}. As $\lambda_\fh(m)= \lambda_{\mathfrak{H}}(m)-\alpha\lambda_\f(m)$, we establish \eqref{ineq} for $\fh$. 
\item Case 2: $c^2= 2 \mu^2$. 
\par
We infer from \eqref{hu4} and \eqref{u4} that for all integers $n\ge 1$,
\begin{eqnarray*}
4^{\ell/2-1/4} \lambda_\fh(4n) = \mu \lambda_\f(n)+ c \lambda_\fh(n).
\end{eqnarray*}
Let $d= c/4^{\ell/2-1/4}$. This recurrence relation gives
\begin{eqnarray*}
\lambda_\fh(4^Jn)= d^J \lambda_\fh(n) + \frac{\mu}{4^{\ell/2-1/4}}\sum_{1\le j<J} d^j\lambda_\f(4^{J-j}n).
\end{eqnarray*}
Note $\mu^2= 2^{2\ell-2}$, so $|d|= 1$, and \eqref{lb} holds for $\lambda_\f(4^{J-j}n)$. Hence, for any integer $m= tr^24^J$ where $t$ is squarefree and $r$ is odd, 
$$
\lambda_\fh(m)\ll |\lambda_\fh(tr^2)| + J^2|\lambda_\f(t)| \tau(r)^2.
$$
By Lemma~\ref{lsqfree} with $\mathcal{Q}=\{2\}$, we get $|\lambda_\fh(tr^2)|\ll |\lambda_\fh(t)| \tau(r)^2$ 
and consequently
$$
\lambda_\fh(m)
\ll (|\lambda_\fh(t)|+|\lambda_\f(t)|)J^2  \tau(r)^2
\ll (|\lambda_\fh(t)| +|\lambda_\f(t)|)\tau(r2^J)^2.
$$ 
\end{itemize}

In summary, we have proved the following. 

\begin{lemma}\label{lem2.3} Let $\f$ be a complete Hecke eigenform, $\g$ and $\fh$ be defined as in \eqref{ghdef}. For any integer $m=tr^2$ where $t\ge 1$ is squarefree,  we have
\begin{eqnarray}\label{ineq}
\lambda_f(m)\ll_\f |\lambda_f(t)| \tau(r)^2+|\lambda_\ff(t)| \tau(r)^2\ll_{\f,\varrho} t^\varrho \tau(r)^2
\end{eqnarray}
for $f=\f,\g,\fh$ respectively, where $\varrho$ satisfies \eqref{varrho}. The first implied $\ll$-constant depends only $\f$ and the second implied $\ll$-constant depends at most on $\f$ and $\varrho$.
\end{lemma}
\begin{remark} When $\f$ lies in the Kohnen plus space, 
the Hecke operator $\mathsf{T}^+(2^2):=\frac32 U_4{\rm pr}$ is taken in place of $\mathsf{T}(2^2)$, 
where ${\rm pr}$ is the orthogonal projection onto the plus space, cf. \cite[p. 42-43]{Kohnen1982}. 
If $\f$ is an eigenform of $\mathsf{T}^+(2^2)$ and $\mathsf{T}(p^2)$ for all odd primes $p$, 
then Lemma~\ref{lem2.3} will still be valid. 
Firstly Lemma~\ref{lem-eigen} and \eqref{Shimuralift} hold for $\f$ and hence \eqref{lb}. 
Next we claim \eqref{lb} holds for  $\g$ and $\fh$.  
For $\g$,  \eqref{gu4} holds if \(f\in\cush\), thus $\g$ is a complete Hecke eigenform so \eqref{lb} holds.  
Note $2\mu\fh= \f\vert_{U_4}$ once $\f$ is in the plus space, see \cite[Proposition 2]{Kohnen1980}; 
thus $2\mu \lambda_\fh(m)=\lambda_\f(4m)$, the claim follows from \eqref{lb} for $\f$. 
\end{remark}

\vskip 8mm

\section{A preparation}\label{S4}

We start with the method of proof  in \cite{HKKL2012} for the set-up. 
Meanwhile we amend, for the case $2\,\|\,d$, the functional equation to relate $\f$ with $\g$ (not $\f_{\frac12}$ in \cite{HKKL2012}), cf. \cite[(4.5)]{HKKL2012} and our Remark~\ref{rmki} (i).  
Lastly we indicate the vital components for improvement with a first attempt (see Proposition~\ref{JL} and Remark~\ref{proprmk}).

Define $\mathbbm{1}_{r^2}(n)=1$ if $r^2\mid n$ and $0$ otherwise.
Replace the divisibility condition with additive characters, we can write
$$
\mathbbm{1}_{r^2}(n) 
= \frac{1}{r^2} \sum_{u ({\rm mod}\,r^2)} \ee\bigg(\frac{nu}{d^2}\bigg)
= \frac{1}{r^2} \sum_{d\mid r^2} \sideset{}{^*}\sum_{u ({\rm mod}\,d)} \ee\bigg(\frac{nu}d\bigg),
$$
where ${\sum}^*_{u ({\rm mod}\,d)}$ runs over  $u (\bmod\,{d})$ coprime to $d$. 
Recall $\mu(n)^2=\sum_{r^2\mid n} \mu(r)$.
When $\sigma >1$, one thus has 
\begin{eqnarray}\label{Mfsum}
L_\f^\flat(s)
=\sum_{r=1}^\infty  \frac{\mu(r)}{r^2} 
\sum_{d\mid r^2} \sideset{}{^*}\sum_{u ({\rm mod}\,d)} L_\f(s, u/d)
\end{eqnarray}
where 
\begin{equation}\label{Lf}
L_\f(s, u/d)= \sum_{m\geq 1} \frac{\lambda_{\f}(m) \ee(m u/d)}{m^s}\cdot
\end{equation}
Let us also denote
\begin{equation}\label{Dr}
D_r(s) := r^{-2} \sum_{d\mid r^2} \sideset{}{^*}\sum_{u ({\rm mod}\,d)} L_\f (s,u/d).
\end{equation}
Now each summand $L_\f(s,u/d)$ extends to an entire function  (explained below), so the task is to establish the (uniform) convergence of the series in $r$. Hence this leads to the estimation of $L_\f(s,u/d)$ in terms of $r$. The method of Hulse et al. is to derive the functional  equation of $L_\f(s,u/d)$ and then apply the convexity principle to $D_r(s)$. They gave an estimate for $D_r(s)$ on the line $\sigma=-\vep$ by bounding $L_\f(s,u/d)$ individually. Consequently they proved that 
\begin{equation}\label{hulse}
D_r(s)\ll r^{2-4\sigma+\vep} (1+|\tau|)^{1-\sigma+2\vep} \quad \mbox{ ($-\vep\le \sigma \le 1+\vep$).}
\end{equation}

To obtain the functional equation of $L_\f(s,u/d)$, one considers for rational $q$, 
\begin{equation*}%\label{dirichlet}
\Lambda(\f,q,s) 
:= \int_0^\infty \f(\ic y+q)y^{s+\frac{\ell}{2} - \frac{1}{4}}\frac{\d y}{y}
= \frac{\Gamma(s + \frac{\ell}{2} - \frac{1}{4})}{(2\pi)^{s + \frac{\ell}{2} - \frac{1}{4}}} 
\sum_{m\geq 1} \frac{\lambda_{\f}(m)\ee(m q)}{m^s}\cdot
\end{equation*}
The integral is absolutely convergent for every $s \in \C$. We define $\Lambda(\g,q,s)$ and $\Lambda(\fh,q,s)$ in the same way. 

Let $q=u/d$ where $(u,d)=1$ and $d\ge 1$. By \cite[Lemma 4.3]{HKKL2012}, $\Lambda(\f,u/d,s)$ satisfies  a functional equation in connection with $\Lambda(\f,-\overline{u}/d, 1-s)$ and $\Lambda(\fh,-\overline{4u}/d,1-s)$ respectively according as $4\mid d$ or  $2\nmid d$, where  $x\overline{x}\equiv 1\,(\bmod\,{d})$. 
For the case $2\,\|\,d$, we revise $\f_{\frac12}$ to be $\g$, which causes a minor change of $\Lambda(\f_{\frac12},-\overline{u}/d, 1-s)$ into $\Lambda(\g,-\overline{u}/(4d), 1-s)$. We would unite the three functional equations into one. Let us introduce
\begin{eqnarray}\label{qd}
q_d=d \ \mbox{  or } \ 2d\ \mbox{ according to $4\mid d$ or not,}
\end{eqnarray}
 and the symbols $\lambda(n; d)$ and $\varpi_d(n,v)$ defined as:
 \begin{equation}\label{lambdafd}
{\renewcommand{\arraystretch}{1.8}
\renewcommand{\tabcolsep}{2.5mm}
\begin{tabular}{|c|c|c|c|}
\hline
{}
& $\lambda(n; d)$ 
& $\varpi_d(n,v)$
\\
\hline
$4\mid d$
& $\lambda_{\f}(n)$  
& $\vep_v^{2\ell+1} \big(\tfrac{d}v\big) \ee\big(\tfrac{-nv}{d}\big)$  
\\
\hline
$2\,\|\,d$
& $\lambda_{\g}(n)$  
& $\vep_v^{2\ell+1} \big(\tfrac{d}v\big) \ee\big(\tfrac{-nv}{4d}\big)$
\\
\hline
$2\nmid d$
& $\lambda_{\fh}(n)$  
& $\ic^{\ell+1/2} \vep_d^{-(2\ell+1)} \big(\tfrac{v}d\big) \ee\big(\tfrac{-\overline{4}nv}{d}\big)$  
\\
\hline
\end{tabular}
}
\end{equation}
with $4\overline{4}\equiv 1\,(\bmod\,{d})$. Write 
\begin{equation}\label{gamma}
L_\infty (s) := (2\pi)^{-s} \Gamma\big(s+\tfrac{\ell}2-\tfrac14)
\end{equation}
and 
\begin{equation}\label{Ltilde}
\widetilde{L}_\f(s, v/d) := \sum_{n\ge 1} \lambda(n; d) \varpi_d(n,v) n^{-s}.
\end{equation}
Now we rephrase \cite[Lemma 4.3]{HKKL2012} of Hulse et al. with the above modification for $2\,\|\,d$. 
\begin{lemma}\label{FunctionalEquation}
Let  $\f\in\cush$ where $\ell\ge 1$ be an integer, $d\in \N$ and $(u,d)=1$. Then $L_\f(s,u/d)$  extends analytically to an entire function and satisfies the functional equation: 
\begin{equation}\label{fe}
q_d^s L_\infty (s) L_\f(s,u/d) = \ic^{-(\ell+1/2)} q_d^{1-s} L_\infty(1-s) \widetilde{L}_\f (1-s, v/d)
\end{equation}
where  $uv\equiv 1\,(\bmod\,{d})$.
\end{lemma}
\begin{remark}\label{rmk1} For the case $2\| d$, the right-side of the equation \eqref{fe} is of period $d$ or probably its divisor in the parameter $v$, which is not obvious in view of the factor $\ee(-\frac{nv}{4d})$. 
Indeed, one checks that $\varpi_d(n,v+d) = \varpi_d(n,v)$ by using (i) if $d=2h$ where $h$ is odd, 
then $\left(\tfrac{d}v\right)= \left(\tfrac{-2}h\right)\left(\tfrac{h}v\right)$; (ii) $n\equiv (-1)^\ell\,(\bmod\,{4})$ in light of the support of $\{\lambda_\g(n)\}$. 
\end{remark}

Assume $\sigma<0$. Applying the functional equation \eqref{fe} to \eqref{Dr}, we obtain 
\begin{equation}\label{drneg}
D_r(s)
= \ic^{-(\ell+1/2)} r^{-2} \frac{L_\infty(1-s)}{L_\infty(s)} \sum_{d\mid r^2} q_d^{1-2s} 
\sum_{n\ge 1} \frac{\lambda(n; d)} {n^{1-s}} 
\sideset{}{^*}\sum_{u\,({\rm mod}\,d)} \varpi_d(n,v).
\end{equation}
With a change of running index into $v$ (as $uv\equiv 1\,(\bmod\,{d})$), 
we observe from \eqref{lambdafd} that the sum over $u$ mod $d$  is a  particular case of Kloosterman-Sali\'e sums, 
see  \cite[Section 3]{Iwaniec1987}. Immediately we have the Weil bound,
\begin{align}\label{weil}
\sideset{}{^*}\sum_{u\,({\rm mod}\,d)} \varpi_d(n,v)\ll d^{1/2}\tau(d)(d, n)^{1/2}.
\end{align}
But in fact it carries more arithmetic properties, as shown below.

\begin{lemma}\label{ksb}
 For $e\in\{0, 1, 2\}$, $b$ an odd squarefree integer and $(a, 2b)=1$, we have
\begin{equation} \label{charsum}
\sideset{}{^*}\sum_{v ({\rm mod}\,2^ea^2b)} \varpi_{2^e a^2b}(m,v) 
= G_{e,b}(m) a^2 \sum_{f\mid a^2}\frac{\mu(f)}{f}  {\mathbbm{1}_{a^2/f}(m)},
\end{equation}
where $G_{e,b}(m)\ll \sqrt{b}$ with an absolute $\ll$-constant, and $\mathbbm{1}_d(n)=1$ if $d\mid n$ and $0$ otherwise.
(Recall that we are confined to $m\equiv (-1)^\ell\,(\bmod\,{4})$ in the case of $e=1$.)
\end{lemma}

Lemma~\ref{ksb}'s proof is postponed to Section~\ref{S6}. Now we apply \eqref{weil} to give a technically lightweight  improvement on the result \eqref{hulse} of  Hulse et al.

Let $\vep>0$ be small and $\sigma=-\vep$. Applying \eqref{weil} and Stirling's formula to  \eqref{drneg},  it follows that (recalling $q_d=d$ or $2d$)
\begin{align*}
D_r(-\vep +\ic\tau)
& \ll r^{-2} (1+|t|)^{1+\vep} \sum_{d\mid r^2} d^{3/2+3\vep}
\sum_{n\ge 1} |\lambda(n; d)|(d,n)^{1/2} n^{-(1+\vep)}
\\
& \ll (r(1+|t|))^{1+\vep}
\end{align*}
because $|\lambda(n; d)|(n,d)^{1/2}\le |\lambda(n; d)|^2 +(n,d)$, implying that the last summation is 
\begin{eqnarray*}
\ll \sum_{n\ge 1} |\lambda(n; d)|^2n^{-(1+\vep)} + \sum_{\ell|d} \sum_{n\ge 1}  n^{-(1+\vep)}\ll d^\vep.
\end{eqnarray*}
By \cite[Lemma 4.2]{HKKL2012}, we have $D_r(1+\vep +\ic \tau) \ll r^{-2}$. An application of Phragm\'en–Lindel\"of principle  gives 
$$
D_r(\sigma+\ic\tau) \ll r^{1-3\sigma+\vep}(1+|\tau|)^\vep.
$$
To assure the convergence in \eqref{Mfsum}, we require $1-3\sigma <-1$ and hence conclude the following. 

\begin{proposition}\label{JL}  
$L_{\f}^{\flat}(\sigma+\ic\tau)
\ll_{{\mathfrak f}, \varepsilon} (|\tau|+1)^{1-\sigma+2\varepsilon}$  for $\tfrac{2}{3}+\varepsilon\le\sigma\le 1+\vep$
and $\tau\in \R$.
\end{proposition}

\begin{remark}\label{proprmk} 
We have applied only the mean square estimate \eqref{msq} for $\g$ and $\fh$, and only the Hecke eigenform property of $\f$ is used.  In the next section, we will invoke the arithmetic property revealed in \eqref{charsum}, the  eigenform properties of all $\f,\g,\fh$ and the approximate functional equation to prove the main result.  
\end{remark}

\vskip 8mm

\section{Proof of Theorem \ref{main}}\label{S5}

We begin with the approximate functional equation for $L_\f(s,u/d)$ below, 
whose proof is given in Section~\ref{S7}. 

\begin{lemma}\label{AFE} 
Let $T\ge 1$ be any number and  $s=\sigma+\mathrm{i}\tau$.  
Suppose $\frac12\le \sigma \le \frac32$ and $|\tau|\le T$. We have
\begin{align*}
L_\f(s, u/d) 
& = \sum_{m\ge 1} \frac{ \lambda_\f(m) \ee(mu/d) }{m^s} V\bigg(\frac{m}{q_dT}\bigg) 
\\
& \quad + \ic^{-(\ell+1/2)}(q_dT)^{1-2s} \sum_{m\ge 1} \frac{ \lambda_{\f,d}(m) \varpi_d(m,v)}{m^{1-s}} V_{s,T}\bigg(\frac{m}{q_dT}\bigg)
\end{align*}
where $uv\equiv 1\,(\bmod\,{d})$, $V(y)$ and $V_{s,T}(y)$ are smooth functions on $(0,\infty)$ 
and satisfy the following: for any $0<\eta<\frac14$,
\begin{align*}
V(y)
& = 1+O_\eta(y^\eta), 
\\
V_{s,T}(y) 
& = \frac{L_\infty(1-s)}{T^{1-2s} L_\infty(s)} +O_\eta(y^\eta)\ll_\eta 1+y^\eta,
\end{align*}
and for any $\eta>0$, both $V(y)$ and $V_{s,T}(y)\ll_\eta y^{-\eta}$.
\end{lemma}

Now we deal with $L_\f^\flat(s)$. In \eqref{Mfsum}, we replace the even squarefree $r$ by  $2r$ and thus 
$$
L_\f^\flat(s)
= \sum_{\substack{r\ge 1\\ {\rm odd}}} \mu(r) D_r(s)  
- \frac14 \sum_{\substack{r\ge 1\\ {\rm odd}}} \mu(r) D_{2r}(s).
$$
Next  we separate the sum over $d$  according as $4\mid d$, $2\,\|\,d$ or $2\nmid d$, 
and hence obtain a decomposition of $L_\f^\flat(s)$ into  three pieces,
$$
L_\f^\flat(s) = M_\infty(\f, s) + M_{1/2}(\f, s) + M_0(\f, s),
$$
where 
\begin{align*}
M_\infty(\f, s)
& := -\frac14\sum_{\substack{r=1\\ {\rm odd}}}^\infty  \frac{\mu(r)}{r^2}  \sum_{d\mid r^2} 
\sideset{}{^*}\sum_{u ({\rm mod}\,4d)} L_\f( s, u/4d),
\\
M_{1/2}(\f, s)
& := -\frac14 \sum_{\substack{r=1\\ {\rm odd}}}^\infty  \frac{\mu(r)}{r^2} \sum_{d\mid r^2} 
\sideset{}{^*}\sum_{u ({\rm mod}\,2d)} L_\f(s, u/2d),
\\
M_{0}(\f, s)
& := \frac34 \sum_{\substack{r=1\\ {\rm odd}}}^\infty \frac{\mu(r)}{r^2}  \sum_{d\mid r^2} 
\sideset{}{^*}\sum_{u ({\rm mod}\,d)} L_\f(s, u/d).
\end{align*}
We shall verify the uniform convergence for the three series of holomorphic functions in $\re s>\frac12$, and concurrently obtain  the desired upper estimate (\ref{UBMs}).

Let $\sigma_0=\frac12+\vep_0$ where $\vep_0>0$ is arbitrarily small but fixed,  and $T\ge 1$ be any integer. Consider $s=\sigma+\ic\tau$ where $\sigma_0\le \sigma\le \sigma_0+\frac12$ and $T-1\le |\tau|\le T$. 
In view of the condition $d\mid r^2$ for squarefree $r$, we decompose into $d=a^2b$ and $r=abc$ where $a,b,c$ are pairwise coprime and squarefree. It is equivalent to consider the series
\begin{eqnarray*}
 \sum_{a,b,c }  \frac{\mu(2abc)}{(abc)^2} 
\sideset{}{^*}\sum_{u ({\rm mod}\,2^e a^2b)} L_\f\Big(s, \frac{u}{2^e a^2b}\Big)
\end{eqnarray*}
where $e=0,1,2$. 
Now we apply Lemma~\ref{AFE} and observe, as before, 
the set of $v$ given by $uv\equiv 1\,(\bmod\,{d})$ runs through a reduced residue class as $u$ varies. We are led to 
\begin{align}
\Sigma_1 
& = \sum_{a,b,c}  \frac{\mu(2abc)}{(abc)^2} \sum_{m\ge 1} \frac{\lambda_\f(m)}{m^s} V\bigg(\frac{m}{a^2bT_e}\bigg) \sideset{}{^*}\sum_{u\,({\rm mod}\,2^ea^2b)} \ee\bigg(\frac{mu}{2^ea^2b}\bigg), 
\label{sigma1}
\\
\Sigma_2   
& = T_e^{1-2s} \sum_{a, b, c} \frac{\mu(2abc)}{a^{4s}b^{1+2s}c^2}   
\sum_{m\ge 1} \frac{\lambda_{\f,e}(m)}{m^{1-s}} V_{s,T}\bigg(\frac{m}{a^2bT_e}\bigg)  
\sideset{}{^*}\sum_{v ({\rm mod}\,2^ea^2b)} \varpi_{2^e a^2b}(m,v),
\label{sigma2}
\end{align}
where $\lambda_{\f,e}(m):=\lambda_{\f, 2^e}(m)$, see \eqref{lambdafd}, and $T_e= 2T$ or $4T$ according as $e=0$ or not (so that $q_{2^ea^2b}T= a^2bT_e$, see \eqref{qd}).

Inserting (\ref{charsum}) into $\Sigma_2$ in (\ref{sigma2}), 
we further decompose $a=fg$ and $m=fg^2h$ in light of the squarefreeness of $f$ 
and the conditions $f\mid a^2$ and $(a^2/f)\mid m$. 
\begin{equation}\label{sigma2a}
\Sigma_2
= T_e^{1-2s} \sum_{f,g,b,c }  \frac{\mu(f)\mu(2fgbc)}{f^{3s}g^{2s}b^{1+2s}c^2}   
\sum_{h\ge 1} \frac{ \lambda_{\f,e}(fhg^2)}{h^{1-s}} V_{s,T}\bigg(\frac{h}{fbT_e}\bigg)G_{e,b}(fhg^2).
\end{equation}
To justify the uniform convergence, it suffices to consider the sum over dyadic ranges: 
$(f,g,b,c)\sim (F,G,B,C)$, meaning $F\le f<2F$, etc. 
Denote by $\Sigma_2^{F,G,B,C}$ the expression on the right-side of (\ref{sigma2a}) under this range restriction. 
We estimate each summand trivially with the bound $G_{e,b}(m)\ll \sqrt{b}$ in Lemma~\ref{ksb}. 
A little simplification leads to
\begin{equation}\label{sigma2a1}
\Sigma_2^{F,G,B,C}
\ll T_e^{1-2\sigma} \sum_{\substack{(f,g, b,c)\\ \sim (F,G,B,C)}}  
\frac{|\mu(2fgbc)|}{f^{3\sigma}g^{2\sigma} b^{1/2+2\sigma}c^2}   
\sum_{h\ge 1} \frac{|\lambda_{\f,e}(fhg^2)|}{h^{1-\sigma}} \left|V_{s,T}\bigg(\frac{h}{fbT_e}\bigg)\right|.
\end{equation}

Next we treat the sum over $h$ in order for the following estimate\footnote{\mbox{}  Throughout the proof, $\vep$ denotes an arbitrarily small positive number whose value may differ, up to our disposal,  at each occurrence.}:
\begin{equation}\label{hsum}
\sum_{h\ge 1} \frac{ |\lambda_{\f,e}(fhg^2)|}{h^{1-\sigma}} \left|V_{s,T}\bigg(\frac{h}{fbT_e}\bigg)\right|
\ll G^\vep (TFB)^{\sigma-1/2+\vep} \sum_{h\le (FBT)^{1+\vep}} \frac{ |\lambda_{\f,e}(fh)|}{\sqrt{h}}\cdot
\end{equation}
To establish \eqref{hsum}, we invoke Lemmas~\ref{lem-eigen} and \ref{lsqfree}, to remove $g$ inside $\lambda_{\f,e}(fhg^2)$, and the estimate for $V_{s,T}$. Set $Q$ to be the set of all primes {\it not} dividing $g$, and write $h=qr^2$ where $p^2|q$ implies $p\in \mathcal{Q}$ and $p|r$ implies $p\notin \mathcal{Q}$ (see Remark~\ref{rmk2}). As $(2f,g)=1$, $Q$ contains $2$ and all the prime factors of $f$. Thus $p^2|fq$ implies $p\in \mathcal{Q}$. Thus, $ |\lambda_{\f,e}(fhg^2)|=|\lambda_{\f,e}(fq (gr)^2)|\ll_\vep  |\lambda_{\f,e}(fq)|(gr)^\vep$. 
From Lemma~\ref{AFE}, we deduce the estimate
$$
V_{s,T}\bigg(\frac{h}{fbT_e}\bigg)\ll_\vep \left\{\begin{array}{ll} 
(FBT)^\vep & \mbox{ for $h\le (FBT)^{1+\vep}$,}\vspace{2mm}\\
 h^{-2} & \mbox{ otherwise.}
\end{array}
\right.
$$
The sum over $h\ge (FBT)^{1+\vep}$ is negligible, in fact $\ll (TFGB)^\vep$ (for which we may use the crude bound  $|\lambda_{\f,e}(fq)|\ll (fq)^{1/2}$ by \eqref{msq}). Consequently, the left side of \eqref{hsum} is 
\begin{align*}
& \ll (FBT)^\vep \sum_{qr^2 \le (FBT)^{1+\vep}} g^\vep r^{2(\sigma-1)+\vep} 
|\lambda_{\f,e}(fq)| q^{-(1-\sigma)}
\\
& \ll G^\vep (FBT)^{\sigma-1/2+\vep} \sum_{q \le(FBT)^{1+\vep}} |\lambda_{\f,e}(fq)| q^{-1/2}
\end{align*}
(recalling $\sigma\ge \sigma_0>1/2$) which is \eqref{hsum} after renaming $q$ into $h$.

Inserting \eqref{hsum} into \eqref{sigma2a1}, we deduce that 
$$
\Sigma_2^{F,G,K,L,C}
\ll (TFGB)^\vep
T^{-\sigma+1/2} F^{-2\sigma} G^{ -2\sigma+1} B^{ -\sigma } C^{-1}
\sum_{f\sim F} \sum_{h\le (FBT)^{1+\vep}} |\lambda_{\f,e}(fh)| (fh)^{-1/2}.
$$
Write $m=fg$ and note the divisor function $\tau(m)\ll_\vep m^\vep$. The  double sum is 
\begin{align*}
& \ll_\vep (F^2BT)^\vep \sum_{m\ll (F^2BT)^{1+\vep}} |\lambda_{\f,e}(m)| m^{-1/2} 
\ll_\vep (F^2BT)^{1/2+\vep},
\end{align*}
by \eqref{msq}. In summary, we get
$$
\Sigma_2^{F,G,K,L,C}
\ll_\vep  
(TFGB)^\vep
T^{1-\sigma} F^{1-2\sigma}G^{1-2\sigma} B^{ -\sigma+1/2} C^{-1}.
$$
Recall $\sigma_0=\frac12+\vep_0$ and take $\vep\le \vep_0/2$. 
Consequently, uniformly for $\sigma_0\le \sigma\le \sigma_0+\frac12$  and $T-1\le |\tau|\le T$, 
we have $\Sigma_2^{F,G,K,L,C}\to 0$ as $\max(F,G,B,C)\to \infty$, concluding the uniform convergence. Moreover, as $T^{1-\sigma} \ll (1+|s|)^{1-\sigma}$, it follows that
\begin{eqnarray*}
\Sigma_2 
\ll \sum_{F,G,B,C} \Sigma_2^{F,G,K,L,C}
\ll_\vep (1+|s|)^{1-\sigma+\vep_0},
\end{eqnarray*}
recalling the multiple summations range over powers of two. 

We turn to $\Sigma_1$ in (\ref{sigma1}) which is plainly treated in the same fashion and indeed easier. The inner exponential sum in (\ref{sigma1}) equals
$$
2^ea^2b \sum_{\delta\mid 2^ea^2b} \frac{\mu(\delta)}{\delta} \mathbbm{1}_{2^ea^2b/\delta}(m).
$$
Noting that $a,b$ are squarefree and $(a,b)=1$, we write $\delta=2^j fk$, $a=fg$ and $b=kl$ where $j=0$ or $1$. Then the summation over $m$ will be confined to run over the sequence of $m=2^{e-j} fg^2l h$ for positive integers $h$. Explicitly we have
\begin{equation}\label{sigma1a}
\Sigma_1 
= \sum_{j=0,1} \sum_{f,g,k,l,c }  \frac{2^{(e-j)(1-s)}\mu(2^j fk)\mu(2fgklc)}{f^{1+s}g^{2s}k^2 l^{1+s}c^2}   
\sum_{h\ge 1} \frac{ \lambda_{\f}(2^{e-j}flhg^2)}{h^s} V\bigg(\frac{h}{2^j fkT}\bigg).
\end{equation}
Analogously we divide the summation ranges into dyadic intervals and consider the subsum of $\sum_{f,g,k,l,c}$ with $(f,g,k,l,c)\sim (F,G, K, L, C)$. Repeating the above argument\footnote{\mbox{}  For the calculation as in \eqref{hsum}, there is a little variant since the exponent $\sigma$ of $|h^s|$ is $>\frac{1}{2}$.},  correspondingly we obtain
\begin{align*}
& \Sigma_1^{F,G,K,L,C}
\\
& \ll_\vep (TFGK)^\vep
F^{-\sigma-1} G^{ 1-2\sigma} K^{ -1}L^{-\sigma-1} C^{-1}  \max_{j=0,1,2}\sum_{(f,l) \sim (F,L)}  
\sum_{h\ll (FKT)^{1+\vep}} |\lambda_\f(2^jflh)| h^{-\sigma}
\\
& \ll_\vep (TFGKL)^\vep
F^{-1} G^{1-2\sigma} K^{ -1}L^{-1} C^{-1}
\sum_{m\ll (F^2KLT)^{1+\vep}} |\lambda_\f(m)| m^{-\sigma}
 \\
& \ll_\vep (TFGKL)^\vep
T^{1-\sigma}F^{1-2\sigma} G^{1-2\sigma} K^{ -\sigma}L^{-\sigma} C^{-1},
\end{align*}
which assures the uniform convergence and the upper estimate. Our proof is complete by changing $\vep_0$ into $2\vep$.

\vskip 8mm

\section{Proof of Lemma~\ref{ksb}}\label{S6}

First consider $e=1$ or $2$, and take the complex conjugate of the left side  to simplify a bit the exponential factor. Then by \eqref{lambdafd},
\begin{eqnarray}\label{sumv}
\sideset{}{^*}\sum_{v ({\rm mod}\,2^ea^2b)} \overline{ \varpi_{2^ea^2b}(m,v)}
= \sideset{}{^*}\sum_{v ({\rm mod}\,2^ea^2b)} \vep_v^{-(2\ell+1)} \bigg(\frac{2^ea^2b}{v}\bigg) 
\ee\bigg(\frac{mv}{2^{4-e}a^2b}\bigg).
\end{eqnarray}
We write $v=\alpha 8b + \beta a^2$. Note that $v$ runs over a reduced residue class mod $2^ea^2b$ when $\alpha$ (mod $a^2$) and $\beta$ (mod $2^eb$) run over the respective reduced residue classes, since $a$ is odd and $(a,2b)=1$. Our substitution choice implies $v\equiv \beta$ (mod $4$) and thus $\vep_v=\vep_\beta$. Moreover, the extended Jacobi symbol may be written as, cf. \cite[p.442 (ii)-(iv)]{Shimura1973},
$$
\bigg(\frac{2^ea^2b}{v}\bigg)
= \bigg(\frac{2^eb}{\beta}\bigg)
= \bigg(\frac{(-1)^{(b-1)/2}2^e}{\beta}\bigg) \bigg(\frac{(-1)^{(b-1)/2}b}{\beta}\bigg)
=  \psi_{2,b}(\beta)\chi_{b'}(\beta),
\quad
(\mbox{say}),
$$
where $b' = (-1)^{(b-1)/2} b$ (is a quadratic discriminant) and $\chi_{b'}(\cdot)$ is the primitive quadratic character of  conductor $b$. (Note  $b$ is odd squarefree.)
Thus, we express the right side of \eqref{sumv} as
$$
G_{e,b}'(m) \sideset{}{^*}\sum_{\alpha\,({\rm mod}\,a^2)} \ee\bigg(\frac{2^{e-1} m\alpha}{a^2}\bigg)
= G_{e,b}'(m) a^2 \sum_{f\mid a^2} \frac{\mu(f)}{f} \mathbbm{1}_{a^2/f}(m)
$$
(cf. \cite[p.44 (3.2)]{IK2004} and recalling $a$ is odd)  where 
\begin{equation}\label{Gsum}
G_{e,b}'(m)
= \sideset{}{^*}\sum_{\beta \,({\rm mod}\,2^eb)} \vep_\beta^{-(2\ell +1)} \psi_{2,b}(\beta)\chi_{b'}(\beta)
\ee\bigg(\frac{m\beta}{2^{4-e}b}\bigg).
\end{equation}
This gives \eqref{charsum} with $G_{e,b}(m)=\overline{G_{e,b}'(m)}$, 
and thus it remains to show $G_{e,b}'(m)\ll \sqrt{b}$ so as to finish the proof.

We separate the sum in (\ref{Gsum}) into two subsums,  whose summands take the same value of $\vep_\beta$, as follows: 
$$
\sideset{}{^*}\sum_{\substack{\beta \,({\rm mod}\,2^eb)\\ \beta\equiv 1 ({\rm mod}\,4)}} 
+ \, \ic^{-(2\ell+1)} 
\sideset{}{^*}\sum_{\substack{\beta \,({\rm mod}\,2^eb)\\ \beta\equiv -1 ({\rm mod}\,4)}}.
$$
With the primitive character $\chi_4$ mod $4$ (given by $\chi_4(n)=(-1)^{(n-1)/2}$ for odd $n$), 
we relax the extra conditions with the factors $\frac12(1+\chi_4(\beta))$ and $\frac12(1-\chi_4(\beta))$. 
Consequently, letting $\theta_{\pm} = \frac12(\ic^{\ell+1/2}\pm {\ic}^{-(\ell+1/2)})$, we rearrange the terms to have
\begin{equation*}
G_{e,b}'(m)  
= \theta_{+} \sideset{}{^*}\sum_{\beta \,({\rm mod}\,2^eb)} 
\psi_{2,b}(\beta)\chi_{b'}(\beta)\ee\bigg(\frac{m\beta}{2^{4-e}b}\bigg)
+ \theta_{-} \sideset{}{^*}\sum_{\beta \,({\rm mod}\,2^eb)} \psi_{2,b}'(\beta)\chi_{b'}(\beta)
\ee\bigg(\frac{m\beta}{2^{4-e}b}\bigg)
\end{equation*}
where $\psi_{2,b}'(\beta)=\chi_4 \psi_{2,b}$. 
Both $\psi_{2,b}$ and $\psi_{2,b}'$ are characters (not necessarily primitive) modulo $8$. 
Repeating the argument of writing $\beta= 8\beta_1+b\beta_2$, we infer that 
\begin{equation}\label{Gbound}
G_{e,b}'(m) 
\ll \bigg|\mathop{{\sum}^*}_{\beta \,({\rm mod}\,b)} \chi_{b'}(\beta)\ee\bigg(\frac{2^{e-1}m\beta}{b}\bigg)\bigg|
\ll b^{1/2}
\end{equation}
by the primitivity of $\chi_{b'}$, see \cite[p.47 (3.12) and p.48 (3.14)]{IK2004}. 

Next we come to the case $e=0$. 
In this case, we set $v= b\alpha +a^2\beta$ with $\alpha\,(\bmod\,{a^2})$, $(\alpha, a)=1$ and $\beta\,(\bmod\,{b})$, 
$(\beta,b)=1$, then   
$$
\mathop{{\sum}^*}_{v ({\rm mod}\,a^2b)}  \bigg(\frac{v}b\bigg) \ee\bigg(\frac{-\overline{4}mv}{a^2b}\bigg)
= a^2 \sum_{f\mid a^2}\frac{\mu(f)}{f} \mathbbm{1}_{a^2/f}(m) 
\sideset{}{^*}\sum_{\beta ({\rm mod}\,b)} \chi_{b'}(\beta) \ee\!\left(-\frac{\overline{4} m\beta}{b}\right).
$$
Take $G_{0,b}(m)$ to be the product of $  \ic^{\ell+1/2} \vep_d^{-(2\ell+1)}$ and the character sum (over $\beta$). This gives, with \eqref{Gbound}, the desired result in \eqref{charsum}, completing the proof.

\vskip 8mm

\section{Proof of Lemma~\ref{AFE}}\label{S7}

Let $H(z)$ be an entire function such that $H(z)\ll_{\eta,A} (1+|z|)^{-A}$ for $\re z=\eta$ and any $A>0$, $H(0)=1$ and $H(z)=H(-z)$. (See \cite{Ha2002} for its construction.) We infer with the residue theorem that 
\begin{align*}
L_\f(s, u/d)
& = \frac1{2\pi \ic} \bigg\{\int_{(2)} - \int_{(-2)}\bigg\} L_\f(s+z, u/d) (q_dT)^z H(z)\frac{\dd z}{z}
=: I_1+I_2, 
\qquad
\mbox{(say)}. 
\end{align*}
Changing $z$ to $-z$ and invoking the functional  equation, we transform $I_2$ into
\begin{eqnarray*}
 \ic^{-(\ell+1/2)}  {(q_dT)^{1-2s}}\frac1{2\pi \ic} \int_{(2)} 
 \widetilde{L}_\f(1-s+z, v/d) \frac{ L_\infty(1-s+z)}{T^{1-2s+2z}L_\infty(s-z)} (q_dT)^{z} H(z)\frac{\dd z}{z}\cdot
\end{eqnarray*}

Set 
\begin{align*}
V(y)
& := \frac1{2\pi \ic} \int_{(2)} y^{-z}H(z)\,\frac{\dd z}{z},
\\
V_{s,T}(y)
& := \frac1{2\pi \ic} \int_{(2)} \frac{ L_\infty(1-s+z)}{T^{1-2s+2z}L_\infty(s-z)} y^{-z}H(z)\,\frac{\dd z}{z}\cdot
\end{align*}
The required formula follows readily after inserting the Dirichlet series of $L_\f(s,u/d)$ and $\widetilde{L}_\f(s,v/d)$ in \eqref{Lf} and \eqref{Ltilde}. 

It remains to check the properties of $V(y)$ and $V_{s,T}(y)$. 
The case of $V(y)$ is quite obvious, and for $V_{s,T}(y)$, we recall the estimate in \cite[Lemma 3.2]{Ha2002}: For $\alpha >-\sigma$, 
$$
\frac{\Gamma(z+\sigma)}{\Gamma(z)}\ll_{\alpha,\sigma} |z+\sigma|^\sigma
\qquad 
(\re z\ge \alpha).
$$
Recalling \eqref{gamma}, this yields
\begin{equation}\label{Lestimate}
\frac{ L_\infty(1-s+z)}{T^{1-2s+2z}L_\infty(s-z)}
\ll_\eta \bigg|\frac{1-s+z+\frac{\ell}2-\frac14}T\bigg|^{1-2\sigma+2\eta}
\ll \left(1+\frac{|z|}T\right)^{1-2\sigma+2\eta}.
\end{equation}
We shift the line of integration to the right, yielding $V_{s,T}(y)\ll_\eta  y^{-\eta}$ for any $\eta >0$ and shift to the left to derive
$$
V_{s,T}(y) = \frac{L_\infty(1-s)}{T^{1-2s} L_\infty(s)} +O_\eta(y^{\eta})
$$
for any  $0< \eta <\frac14$. The main term is $O(1)$ by (\ref{Lestimate}). The proof of Lemma~\ref{AFE} ends. 

\vskip 3mm

{\bf Acknowledgments}.
Lau is supported by GRF 17302514 of  the Research Grants Council of Hong Kong.
L\"u is supported in part by the key project of the National Natural Science Foundation of China (11531008) and IRT1264.
The preliminary form of this paper was finished during the visit of E. Royer and J. Wu at the University of Hong Kong in 2015. 
They would like to thank the department of mathematics for hospitality and excellent working conditions.


\vskip 10mm

\begin{thebibliography}{CC}

\bibitem{ConreyIwaniec2000} 
J. B. Conrey \& H. Iwaniec, 
\textit{The cubic moment of central values of automorphic $L$-functions}, 
Ann. Math. {\bf 151} (2000), 1175--1216.

\bibitem{Ha2002}
G. Harcos, 
\textit{Uniform approximate functional equation for principal $L$-functions}, 
IMRN {\bf 18} (2002), 923--932.

\bibitem{HKKL2012} 
T. A. Hulse, E. M. Kiral, C. I. Kuan \& L.-M. Lim, 
\textit{The sign of Fourier coefficients of half-integral weight cusp forms}, 
Int. J. Number Theory {\bf 8} (2012),  749--762.

\bibitem{Iwaniec1987} 
H. Iwaniec, 
\textit{Fourier coefficients of modular forms of half-integral weight}, 
Invent. Math. {\bf 87} (1987), 385--401.

\bibitem{Iwaniec1997} 
H. Iwaniec, 
\emph{Topics in classical automorphic forms}. 
Graduate Studies in Mathematics, 17. American Mathematical Society, Providence, RI, 1997. 

\bibitem{ILS2000} 
H. Iwaniec, W. Luo \& P. Sarnak,
\textit{Low lying zeros of families of $L$-functions}. 
Inst. Hautes Études Sci. Publ. Math.  {\bf 91}  (2000), 55--131. 

\bibitem{IK2004} 
H. Iwaniec \& E. Kowalski,
\emph{Analytic number theory}. 
American Mathematical Society Colloquium Publications, 53. American Mathematical Society, Providence, RI, 2004.

\bibitem{Kohnen1980} 
W. Kohnen, 
\textit{Modular forms of half-integral weight on $\Gamma_0(4)$}, 
Math. Ann. {\bf 248} (1980), 249--266.

\bibitem{Kohnen1982} 
W. Kohnen, 
\textit{Newforms of half-integral weight}, 
J. Reine Angew. Math. {\bf 333} (1982), 32--72.

\bibitem{LauRoyerWu2014}
Y.-K. Lau, E. Royer \& J. Wu,
\textit{Sign of Fourier coefficients of modular forms of half integral weight},
Mathematika, to appear.

\bibitem{MR}
K. Matom\"aki \& M. Radziwill, 
\textit{Multiplicative functions in short intervals},
Ann. of Math., to appear.

\bibitem{Niwa1977} 
S. Niwa, 
\textit{On Shimura's trace formula},
Nagoya Math. J.  {\bf 66}  (1977), 183--202. 

\bibitem{Ono2004}
K. Ono, 
\emph{The web of modularity: arithmetic of the coefficients of modular forms and {$q$}-series}, 
CBMS Regional Conference Series in Mathematics, vol. {\bf 102}, 
American Mathematical Society, Providence, RI, 2004.

\bibitem{Purkait2013} 
S. Purkait, 
\textit{On Shimura's decomposition},
Int. J. Number Theory  {\bf 9}  (2013),  1431--1445. 

\bibitem{Shimura1973} 
G. Shimura, 
\textit{On modular forms of half-integral weight}, 
Ann. Math. {\bf 97} (1973), 440--481.

\end{thebibliography}
\end{document}